\documentclass[]{amsart}

\usepackage[dvipdfmx]{graphicx,color}
\usepackage{amsmath,amssymb,amsthm}

\makeatletter
\@addtoreset{equation}{section}

\makeatother

\newtheorem{theorem}{Theorem}[section]

\newtheorem{corollary}[theorem]{Corollary}
\newtheorem{lemma}[theorem]{Lemma}
\newtheorem*{conjecture*}{Conjecture}
\newtheorem*{maintheorem*}{Main Theorem}
\newtheorem*{modifiedconjecture*}{Modified Conjecture}
\theoremstyle{definition}

\newtheorem{hypothesis}[theorem]{Hypothesis}
\theoremstyle{remark}
\newtheorem{remark}[theorem]{Remark}

\DeclareMathOperator{\di}{diam}

\begin{document}

\title[Geometric inequalities involving mean curvature]{Geometric inequalities involving mean curvature for closed surfaces}
\author{Tatsuya Miura}
\address{Tokyo Institute of Technology, 2-12-1 Ookayama, Meguro-ku, Tokyo 152-8551, Japan}
\email{miura@math.titech.ac.jp}

\maketitle

\begin{abstract}
  In this paper we prove some geometric inequalities for closed surfaces in Euclidean three-space.
  Motivated by Gage's inequality for convex curves, we first verify that for convex surfaces the Willmore energy is bounded below by some scale-invariant quantities.
  In particular, we obtain an optimal scaling law between the Willmore energy and the isoperimetric ratio under convexity.
  In addition, we address Topping's conjecture relating diameter and mean curvature for connected closed surfaces.
  We prove this conjecture in the class of simply-connected axisymmetric surfaces, and moreover obtain a sharp remainder term which ensures the first evidence that optimal shapes are necessarily straight even without convexity.
\end{abstract}

\maketitle


\section{Introduction}\label{sec:intro}

Throughout this paper, we call a closed surface $\Sigma$ smoothly immersed into $\mathbb{R}^3$ simply a {\em surface}, if not specified.
The purpose of this paper is to obtain some geometric inequalities involving mean curvature in some classes of surfaces.
The contents are mainly two-fold: The first part aims at extending Gage's classical isoperimetric inequality for convex curves to surfaces: The second one is devoted to Topping's conjecture relating mean curvature and diameter.

\subsection{Gage-type inequalities for convex surfaces}

A classical isoperimetric inequality by Gage \cite{Gage1983} asserts that
\begin{equation}\label{eqGage}
  \int_\gamma\kappa^2 \geq \pi \frac{L}{A}
\end{equation}
holds for every convex Jordan curve $\gamma$ in $\mathbb{R}^2$, where $\kappa$, $L$, $A$ denote the curvature, the length, and the enclosed area, respectively.
The equality is attained only by a round circle.
Inequality \eqref{eqGage} multiplied by $L$ relates the normalized bending energy and the isoperimetric ratio, which are different order measurements of roundness, as both are scale invariant and minimized by round circles.
As a corollary we deduce that {\em every curve shortening flow of convex curves decreases the isoperimetric ratio}.
This is a direct consequence of \eqref{eqGage} and the derivative formula along a curve shortening flow $\{\gamma_t\}_{t\in[0,T)}$:
$$\frac{d}{dt}\frac{L^2}{A} = -2\frac{L}{A} \Big( \int_{\gamma_t}\kappa^2 - \pi \frac{L}{A} \Big).$$
The convexity assumption in \eqref{eqGage} is necessary due to a dumbbell-like curve with a long thin neck, for which the length can be solely large.
See also recent progress \cite{Bucur2017,Ferone2016} in which a weaker inequality is established even for nonconvex curves.

In this paper we first aim at extending Gage's inequality to convex surfaces.
To this end we first look at the behavior of the isoperimetric ratio under mean curvature flow, namely a one-parameter family of smooth closed surfaces $\{\Sigma_t\}_{t\in[0,T)}$ whose normal velocity coincides with the mean curvature.
Standard first variation formulae imply that the isoperimetric ratio $I=A^{\frac{3}{2}}/V$ satisfies
\begin{equation}\label{eqderivative}
  \frac{dI}{dt} = -\frac{A^\frac{1}{2}}{V} \Big( 3\int_{\Sigma_t}H^2 - \frac{A}{V}\int_{\Sigma_t}H \Big),
\end{equation}
where $A$ denotes the surface area, $V$ the enclosed volume, and $H$ the inward mean curvature scalar, defined by the {\em average} of principle curvatures so that $H\equiv 1$ for the unit sphere.
Note that for round spheres the right-hand side of \eqref{eqderivative} is always zero, and this is compatible with the fact that round spheres are self-shrinkers.

Our main result ensures that for convex surfaces the so-called Willmore energy $\int H^2$ can be related with the remaining term $\frac{A}{V}\int H$ in a similar way to \eqref{eqGage}.

\begin{theorem}\label{thm:main1}
  There exists a universal constant $C\geq4$ such that
  \begin{equation}\label{eqmain1}
     C \int_\Sigma H^2 \geq \frac{A}{V}\int_\Sigma H
  \end{equation}
  holds for every convex surface $\Sigma\subset\mathbb{R}^3$.
  In addition, for every $C<4$ there exists a convex surface for which \eqref{eqmain1} does not hold.
\end{theorem}

We thus obtain a higher dimensional version of Gage's inequality up to a universal constant (for example, we can take $C=108\pi$).
On the other hand, we also discover the necessary lower bound $C\geq4$ due to a {\em cigar-like surface} $\Sigma_\varepsilon$, namely a cylinder of radius $\varepsilon\ll1$ and height $1$ capped by hemispheres, which satisfies
\begin{equation*}
  \int_{\Sigma_\varepsilon} H^2 \approx \Big(\frac{\varepsilon^{-1}}{2}\Big)^2 2\pi \varepsilon = \frac{\pi}{2\varepsilon}, \quad \mbox{and} \quad \frac{A}{V}\int_{\Sigma_\varepsilon} H \approx \frac{2\pi\varepsilon}{\pi\varepsilon^2}\left(\Big(\frac{\varepsilon^{-1}}{2}\Big) 2\pi \varepsilon \right)= \frac{2\pi}{\varepsilon}.
\end{equation*}
In particular, $C=3$ is not allowable; this fact with \eqref{eqderivative} highlights the significant difference from curve shortening flow that {\em there exists a convex mean curvature flow that increases the isoperimetric ratio in a short time interval}.
This should be also compared with classical well-known results by Huisken \cite{Huisken1984}, which give several evidences that ``convex mean curvature flows become spherical''; namely, every convex initial surface retains convexity before shrinking to a point in finite time, and a normalized flow converges to a round sphere.
In addition, we should also recall that under mean curvature flow the isoperimetric difference $A^\frac{3}{2}-6\sqrt{\pi}V$ always monotonically decreases, see e.g.\ \cite{Topping1998,Schulze2008}.

Since it turned out that an ``optimal shape'' for \eqref{eqmain1} is not a round sphere, we are now led to seek another form that is potentially optimized by a sphere.
From this point of view it is worth mentioning that, combining \eqref{eqmain1} with Minkowski's inequality (see e.g.\ (24) in \cite[\S20]{Burago1988}, \cite[Notes for Section 6.2]{Schneider1993}, or recent \cite{Agostiniani2019}):
\begin{equation}\label{eqMinkowski}
  \int_\Sigma H \geq \sqrt{4\pi A} \quad \mbox{for convex }\Sigma, 
\end{equation}
we can directly relate the Willmore energy and the isoperimetric ratio $I=A^\frac{3}{2}/V$.

\begin{corollary}\label{cor:isoperimetric}
  There exists a universal constant $C' \geq 3/(2\sqrt{\pi})$ such that
  \begin{equation}\label{eqmain2}
    C' \int_\Sigma H^2 \geq I
  \end{equation}
  holds for every convex surface $\Sigma\subset\mathbb{R}^3$.
  For $C'<3/(2\sqrt{\pi})$ a round sphere does not satisfy \eqref{eqmain2}.
\end{corollary}

The lower bound of $C'$ is due to a round sphere, in contrast to $C$.
In fact, we expect that the nature of \eqref{eqmain2} is quite different from that of \eqref{eqmain1} in view of the optimal constants $C=\sup_\Sigma E$ and $C'=\sup_\Sigma E'$ for convex $\Sigma$, where $E:=(\frac{A}{V}\int H)(\int H^2)^{-1}$ and $E':=I(\int H^2)^{-1}$.
One reason is that for a cigar-like surface $\Sigma_\varepsilon$ with $\varepsilon \ll 1$, we already know $E(\Sigma_\varepsilon)\approx 4>3=E(\mathbb{S}^2)$, while $E'(\Sigma_\varepsilon)=O(\varepsilon^{1/2})\to0$.
Finding the optimal values of $C$ and $C'$ seems out of scope and is left open.
At this time, only $E'$ has potential to be optimized by a round sphere.

Both estimates \eqref{eqmain1} and \eqref{eqmain2} are optimal in view of the scaling laws.
Indeed, for a {\em pancake-like surface} $\Sigma_\varepsilon$ with $\varepsilon\ll1$, namely the surface surrounding the $\varepsilon$-neighborhood of a flat disk, both sides in \eqref{eqmain2} (and hence \eqref{eqmain1}) diverge as $O(\varepsilon^{-1})$.

The convexity assumption in \eqref{eqmain1} and \eqref{eqmain2} is unremovable as in Gage's result.
Indeed, for a well-known example of a nearly double-sphere connected by a catenoid, the left-hand sides in \eqref{eqmain1} and \eqref{eqmain2} diverge, while the Willmore energy remains less than $8\pi$, cf.\ \cite{Schygulla2012,Kuwert2018}.

Estimate \eqref{eqmain2} is meaningful only in the large-deviation regime ($I\gg1$), although Gage's inequality is optimal even for nearly round curves.
A kind of small-deviation counterpart of \eqref{eqmain2} is already obtained by R\"{o}ger-Sch\"{a}tzle \cite{Roeger2012}.
They show that every surface $\Sigma$ with $I(\Sigma)-I(\mathbb{S}^2)\leq\sigma$ (not necessarily convex) satisfies
\begin{equation}\label{eqRS}
  \bar{C}\Big( \int_\Sigma H^2 -\int_{\mathbb{S}^2} H^2 \Big) \geq  I(\Sigma)-I(\mathbb{S}^2)  \quad \mbox{for some}\ \bar{C}(\sigma)>0.
\end{equation}
The presence of $\sigma>0$ is in general necessary due to nearly double-spheres, but Corollary \ref{cor:isoperimetric} now implies that if we assume that $\Sigma$ is convex, then \eqref{eqRS} holds for some universal constant $\bar{C}$ (not depending on $\sigma$).
The proof of \eqref{eqRS} is based on de Lellis-M\"{u}ller's rigidity estimate for nearly umbilical spheres \cite{DeLellis2005}.
We remark that the optimal constant in (a version of) de Lellis-M\"{u}ller's estimate is known if $\Sigma$ is convex \cite{Perez2011} or outward-minimizing \cite{Agostiniani2019}, but in order to know the optimal $\bar{C}$ in \eqref{eqRS} a substantial progress seems necessary even if we assume convexity.

The main issue in proving Theorem \ref{thm:main1} is how to relate the Willmore energy and other quantities.
R\"{o}ger-Sch\"{a}tzle's idea is applicable to nearly umbilical surfaces but not to our large-deviation regime.
Recently, a potential theoretic approach is developed for obtaining several old and new geometric inequalities \cite{Agostiniani2020,Agostiniani2020a,Agostiniani2019}, but it seems not directly applicable to our problems.
In addition, although many inequalities involving total mean curvature are known for convex surfaces, e.g.\ by using the mean-width representation (cf.\ \cite{Burago1988}), much less is known about the Willmore energy.
In particular, we cannot obtain our estimates via the Cauchy-Schwarz inequality $\int H^2 \geq A^{-1}(\int H)^2$ since this is not sharp for pancake-like surfaces.

Our idea is to establish and employ the following estimate for convex surfaces:
\begin{equation}\label{eqdF}
  \di(\Sigma)^{p-2}\int_\Sigma H^p > c_p \mathcal{D}^{p-1} \quad (1\leq p<\infty),
\end{equation}
where $\di(\Sigma)$ denotes the extrinsic diameter, and $\mathcal{D}$ the ``degeneracy'' (defined in Section \ref{sec:diameter-degeneracy}), which is comparable with the minimal width under $\di(\Sigma)=1$.
Theorem \ref{thm:main1} then follows by \eqref{eqdF} with $p=2$ and by the additional estimate that $\mathcal{D}\gtrsim\frac{A}{V}\int H$.
The proof of \eqref{eqdF} is based on a slicing argument with the help of geometric restriction due to convexity.
As a key ingredient we also use the general scaling law $\int |\kappa|^p \gtrsim r^{1-p}$ for each cross section plane curve, where $r$ is the minimal width of the curve.
We finally indicate that our explicit choice of $c_p$ in \eqref{eqdF} is not optimal in general but sharp as $p\to1$, and in particular $c_1=\pi$ agrees with the optimal constant in Topping's conjecture, the details of which are given below.

\subsection{Topping's conjecture for axisymmetric surfaces}

In his 1998 paper \cite{Topping1998} Topping poses the following
\begin{conjecture*}
  Let $\Sigma\subset\mathbb{R}^3$ be an immersed connected closed surface.
  Then
  \begin{equation}\label{eqTopping}
    \frac{1}{\di(\Sigma)}\int_\Sigma |H| > \pi.
  \end{equation}
\end{conjecture*}

The constant $\pi$ cannot be improved due to cigar-like surfaces.
In \cite{Topping2008} Topping himself already proves a modified version of \eqref{eqTopping}, which weakens $\pi$ to $\frac{\pi}{32}$ but strengthens $\di(\Sigma)$ to the intrinsic diameter (and also deals with higher dimensions).
The exact form of \eqref{eqTopping} is classically known for convex surfaces, where a degenerate segment is optimal (see Section \ref{subsec:rigidity}).
To the author's knowledge, the other known case is only for constant mean curvature (CMC) surfaces \cite{Topping1997}.
For CMC surfaces, even $\int_\Sigma |H| \geq 2\pi \di(\Sigma)$ holds true, with equality only for a round sphere, and hence this class does not contain optimal shapes.
Therefore, the unique nature of Topping's conjecture seems not well understood for nonconvex surfaces.

In this paper we gain more insight into Topping's conjecture by focusing on axisymmetric surfaces.
This class is flexible enough to include both nearly optimal and highly nonconvex surfaces.
The assumption of axisymmetry fairly reduces the freedom of surfaces, but certainly keeps substantial difficulties; for example, even for the simplest dumbbell-like surface, the diameter may not be attained in the axial direction; even if attained, the co-area formula in that direction may involve a part where the mean curvature vanishes, so that we cannot directly extract the diameter and do need further quantitative controls.

Our main result, however, gives an affirmative answer to Topping's conjecture for every simply-connected axisymmetric surface (including dumbbells).
In fact, we obtain a stronger assertion by discovering a sharp remainder term.
For a simply-connected axisymmetric surface $\Sigma$, we define a scale-invariant quantity $U$ by
\begin{align*}
  U(\Sigma) := \Big( \int_0^1 \sqrt{1-T(t)\cdot\omega_\mathrm{axis}} dt \Big)^2
\end{align*}
where $\omega_\mathrm{axis}$ is a unit vector parallel to an axis of symmetry $L_\mathrm{axis}$, and $T$ denotes the unit tangent of a constant-speed minimal geodesic $\gamma_\Sigma:[0,1]\to\Sigma\subset\mathbb{R}^3$ in $\Sigma$ connecting a (unique) pair of points in $\Sigma\cap L_\mathrm{axis}$ such that $\gamma_\Sigma(1)-\gamma_\Sigma(0)=\lambda\omega_\mathrm{axis}$ holds for some $\lambda\geq0$.
Note that $\Sigma$ is generated by revolving $\gamma_\Sigma$ around $L_\mathrm{axis}$; we call $\gamma_\Sigma$ a {\em generating curve} of $\Sigma$ as usual.
We also remark that unless $\Sigma$ is a round sphere, $L_\mathrm{axis}$ is unique so that $\omega_\mathrm{axis}$ is unique up to the sign, and $\gamma_\Sigma$ is unique up to the revolution and the choice of parameter-orientation.
In particular, $U(\Sigma)$ is well defined and strictly positive for any given $\Sigma$.
The quantity $U(\Sigma)$ measures a certain deviation from a ``unidirectional'' shape since $U(\Sigma)\ll1$ corresponds to $|T-\omega_\mathrm{axis}|\ll1$ in a certain sense.

Here is our main result concerning Topping's conjecture.

\begin{theorem}\label{thm:Topping}
  There exists a universal constant $\sigma>0$ such that
  \begin{equation}\label{eq:axiTopping}
    \frac{1}{\di(\Sigma)}\int_\Sigma|H| \geq \pi + \sigma U(\Sigma)
  \end{equation}
  holds for every simply-connected axisymmetric surface $\Sigma\subset\mathbb{R}^3$.
\end{theorem}

Theorem \ref{thm:Topping} is sharp in the sense that there is a sequence such that the ratio $\frac{1}{U(\Sigma_\varepsilon)}(\frac{1}{\di(\Sigma_\varepsilon)}\int_{\Sigma_\varepsilon}|H|-\pi)$ converges as $\varepsilon\to0$.
However, for a cigar-like surface $\Sigma_\varepsilon$ this ratio diverges and behaves like $\varepsilon^{-1}$.
One example for which the ratio converges is a ``double-cone'' surface, which is made by connecting the circular bases of two thin cones (see Remark \ref{rem:doublecone}).
This reveals that conical ends are more favorable than round caps at a higher order level.

Estimate \eqref{eq:axiTopping} not only directly verifies Topping's conjecture for a new class of surfaces, but also implies that a minimizing sequence in that class needs to degenerate into a segment, thus giving the first evidence for nonconvex surfaces that optimal shapes are necessarily almost straight.

\begin{corollary}\label{cor:Topping}
  Topping's conjecture \eqref{eqTopping} holds true for every simply-connected axisymmetric surface.
  Moreover, for a sequence $\{\Sigma_n\}_n$ of such surfaces, if
  $$\lim_{n\to\infty}\frac{1}{\di(\Sigma_n)}\int_{\Sigma_n}|H| = \pi,$$
  then up to similarity a sequence $\{\gamma_{\Sigma_n}\}_n$ of generating curves of $\Sigma_n$ converges to a unit-speed segment $\bar{\gamma}:[0,1]\to\mathbb{R}^3$ in the sense of $W^{1,p}$ for every $1\leq p <\infty$.
\end{corollary}

The convergence in $W^{1,p}$ is optimal in the sense that $\gamma_n$ does not converge to $\bar{\gamma}$ in $W^{1,\infty}$ since at the endpoints $\gamma_n$ is perpendicular to an axis of symmetry and hence to $\bar{\gamma}$; even in the interior, $\gamma_n$ may have small loops that vanish as $n\to\infty$.

In the proof of Theorem \ref{thm:Topping}, given an axisymmetric $\Sigma$, we construct a comparison convex surface $\Sigma'$ such that $\int_\Sigma |H| \geq \int_{\Sigma'}|H|$ and $\di(\Sigma')\geq\di(\Sigma)$ by using a rearrangement argument introduced in \cite{Dalphin2016}, in which Minkowski's inequality \eqref{eqMinkowski} is extended to certain axisymmetric surfaces.
Since the mean curvature is already well studied in \cite{Dalphin2016}, our main contribution in this argument is concerning the diameter, which is less tractable due to its nonlocal nature.
In addition, in order to extract the remainder $U$, we need essentially new quantitative controls in the rearrangement procedures.

This paper is organized as follows.
In Section \ref{sec:diameter-degeneracy} we first prove estimate \eqref{eqdF} and then prove Theorem \ref{thm:main1}.
In Section \ref{sec:Topping} we first recall the rearrangement arguments and establish general diameter estimates, and then prove Theorem \ref{thm:Topping}.

\subsection{Notation}

The notation $f\lesssim g$ means that there is a universal $C>0$ such that $f\leq C g$ holds.
We also define $f\gtrsim g$ similarly, and use $f\sim g$ in the sense that both $f\lesssim g$ and $f\gtrsim g$ hold.
In addition, the notation $f\ll g$ in an assumption means that there exists $\varepsilon>0$ such that if $f\leq\varepsilon g$, then the assertion holds.

{\small
\subsection*{Acknowledgments}
The author would like to thank \mbox{Felix Schulze} and \mbox{Peter Topping} for reading an earlier version of this manuscript.
This work is supported by JSPS KAKENHI Grant Numbers 18H03670, 20K14341, and 21H00990, and by Grant for Basic Science Research Projects from The Sumitomo Foundation.
}

\section{Gage-type inequalities for convex surfaces}\label{sec:diameter-degeneracy}

We first define the degeneracy of a convex surface $\Sigma$.
For a unit vector $\omega\in\mathbb{S}^2\subset\mathbb{R}^3$ the width (breath) of $\Sigma$ in the direction $\omega$ is defined by
\begin{equation*}
  b_\Sigma(\omega):=\max\{ (q-q')\cdot\omega \mid q,q'\in\Sigma\subset\mathbb{R}^3 \},
\end{equation*}
where $\cdot$ denotes the inner product in $\mathbb{R}^3$.
The extrinsic diameter of $\Sigma$ is given by
\begin{equation*}
  \di(\Sigma):=\max\{ b_\Sigma(\omega) \mid \omega\in\mathbb{S}^2 \}.
\end{equation*}
Then we define the {\em degeneracy} $\mathcal{D}$ of $\Sigma$ by
\begin{equation}\label{eq40}
  \mathcal{D}(\Sigma):=\max\left\{ \frac{b_\Sigma(\omega_0)}{b_\Sigma(\omega)}\ \left|
  \begin{array}{l}
    \omega,\omega_0\in\mathbb{S}^2,\ \omega\cdot\omega_0=0, \\
    b_\Sigma(\omega_0)=\di(\Sigma)
  \end{array}
  \right\}\right..
\end{equation}
Note that the degeneracy $\mathcal{D}$ is comparable with the ratio (diameter)/(minimal width) up to universal constants, but slightly different as $\omega$ is taken from the orthogonal complement of a diameter-direction $\omega_0$.
Here we adopt this $\mathcal{D}$ for computational simplicity.

Now we are in a position to state \eqref{eqdF} rigorously.

\begin{theorem}\label{thm:diameter-degeneracy}
  Every convex surface $\Sigma\subset\mathbb{R}^3$ satisfies
  \begin{equation}\label{eq11}
    \di(\Sigma)^{p-2}\int_\Sigma H^p > c_p \mathcal{D}^{p-1},
  \end{equation}
  where $c_p$ is a positive constant depending only on $p$.
  In particular, we can take
  $$c_p = 2 \Big( \int_0^\infty(1+t^2)^{\frac{1-3p}{2p}}dt \Big)^p \Big( 2\int_{0}^{1/2}(1+t^{-2})^{\frac{1-p}{2}}dt \Big).$$
\end{theorem}

\begin{remark}
  The above choice yields the optimal constant $c_1=\pi$ only for $p=1$.
  In the case of $p=2$, we have $c_2= 2 (\frac{1}{2}\mathrm{B}(\tfrac{1}{2},\tfrac{3}{4}))^2 (\sqrt{5}-2) = 0.6777700... \geq \frac{2}{3}$, where $\mathrm{B}$ denotes the beta function.
\end{remark}

Our proof of Theorem \ref{thm:diameter-degeneracy} is based on a slicing argument, and hence it is important to gain scale-analytic insight into the curvature energy for plane curves; such a point of view played important roles in previous variational studies of elastic curves, see e.g.\ \cite{Miura2016,Miura2017,Miura2020}.
In this paper we use the fact that each cross section curve of a convex surface has a lower bound (also valid for nonconvex curves).

\begin{lemma}\label{lem:degeneracy}
  Let $p\geq1$.
  For an immersed closed plane curve $\gamma$ bounded by two parallel lines of distance $r>0$, we have
  \begin{align}\label{eq13}
    \int_{\gamma} |\kappa|^p ds \geq \tilde{c}_pr^{1-p},
  \end{align}
  where $\kappa$ denotes the curvature and $s$ the arclength parameter of $\gamma$, and
  $$\tilde{c}_p= 2 \Big( 2\int_0^\infty(1+t^2)^{\frac{1-3p}{2p}}dt \Big)^p
  = 2 \Big( \mathrm{B} \big( \tfrac{1}{2},\tfrac{2p-1}{2p} \big) \Big)^p.
  $$
\end{lemma}

\begin{proof}
  Up to rescaling we may assume that $r=1$.
  In addition, up to a rigid motion, we may assume that $\gamma$ lies in the strip region $[0,1]\times\mathbb{R}$.
  Then the curve $\gamma$ contains at least two disjoint graph curves represented by functions $u_i:(a_i,b_i)\to\mathbb{R}$ ($i=1,2$) with $0\leq a_i\leq b_i\leq 1$ such that $u_i'(c_i)=0$ at some $c_i\in(a_i,b_i)$, and $|u_i'(x)|\to\infty$ both as $x\downarrow a_i$ and as $x\uparrow b_i$; indeed, such graphs are found near the maximum and minimum of $\gamma$ in the vertical direction.
  Dropping the index $i$, we now prove for the graph curve $G_u:=\{(x,u(x)) \in \mathbb{R}^2 \mid x\in[a,b]\}$ that
  \begin{equation}\label{eq14}
    \int_{G_u}|\kappa|^pds \geq \frac{\tilde{c}_p}{2},
  \end{equation}
  which implies \eqref{eq13} after addition with respect to the two graphs.

  We begin with the direct computation that
  $$\int_{G_u}|\kappa|^pds = \int_a^b \Big(\frac{|u''|}{(1+|u'|^2)^{3/2}}\Big)^p\sqrt{1+|u'|^2}dx = \int_a^b \left|\big(f(u')\big)'\right|^pdx,$$
  where $f(t):=\int_0^{t}(1+\tau^2)^{\frac{1-3p}{2p}}d\tau$.
  Applying the H\"{o}lder inequality to the right-hand side, and recalling that $b-a\leq 1$, we have
  $$\int_{G_u}|\kappa|^pds \geq \Big(\int_a^b \left|\big(f(u')\big)'\right|dx\Big)^p.$$
  Now for \eqref{eq14} it suffices to prove
  $$\int_a^b \left|\big(f(u')\big)'\right|dx \geq 2\int_0^\infty f'(t)dt = \Big(\frac{\tilde{c}_p}{2}\Big)^{1/p}.$$
  This follows by decomposing the left-hand side's integration interval at $c\in(a,b)$ (where $u'(c)=0$) and by using, for each of the two integrals, the triangle inequality $\int|(f(u'))'|\geq|\int(f(u'))'|$, the boundary conditions $|u'(a+0)|=|u'(b-0)|=\infty$ and $u'(c)=0$, and also the oddness of $f$.
\end{proof}

In the case of $p=2$, the same kind of lemma is obtained in \cite[Lemma 4.3]{Miura2017}.
In addition, Henrot-Mounjid \cite{Henrot2017} study a closely related problem, which minimizes the same curvature energy with $p=2$ among convex curves of prescribed inradius $r_\mathrm{in}$; the inradius is always bounded by the half-width $r/2$.
The constants in both \cite{Henrot2017,Miura2017} are represented by using $\cos\theta$, but they are in fact the same as our constant $\tilde{c}_p$ with $p=2$ after a change of variables.
In view of this, we can also represent $\tilde{c}_p$ as
$$\tilde{c}_p = 2 \Big( 2\int_0^{\pi/2}(\cos\theta)^{\frac{p-1}{p}}d\theta \Big)^p.$$

\begin{remark}[Optimality of $\tilde{c}_p$]
  Compared to $c_p$ in Theorem \ref{thm:diameter-degeneracy}, the value of $\tilde{c}_p$ is more important because of its optimality.
  Below we briefly argue the optimality, assuming $p>1$; the case of $p=1$ is trivial.
  Let $f$ be as in the proof of Lemma \ref{lem:degeneracy}.
  Let $u:[-1,1]\to\mathbb{R}$ be the primitive function of the increasing function $f^{-1}({A}x)$, where ${A}:=\lim_{t\to\infty}f(t)\in(0,\infty)$, such that $u(0)=0$.
  Then $u$ is a symmetric convex function such that $\lim_{x\to\pm1}|u'(x)|=\infty$, and also $\lim_{x\to\pm1}|u(x)|$ is defined as a finite value because for $x\in(0,1)$ we have
  $$u(x) = \int_0^xf^{-1}({A}y)dy = \frac{1}{{A}}\int_0^{f^{-1}({A}x)}zf'(z)dz,$$
  and $zf'(z)\sim z^{\frac{1-2p}{p}}$ as $z\to\infty$, where the exponent $\frac{1-2p}{p}$ is strictly less than $-1$.
  In addition, we have the identity $(f(u'))'\equiv {A}$ so that in view of the H\"older inequality in the proof of Lemma \ref{lem:degeneracy}, it is straightforward to check that a closed convex curve made by connecting the graph curve of $u$ and its vertical reflection attains the equality in \eqref{eq13} for $r=2$.
  Notice that the resulting closed curve is of class $C^2$ (but not $C^3$); the only nontrivial point is whether the curvature is well defined where the two graph curves are connected, but in fact the curvature vanishes there since $(f(u'))'=|\kappa|^p\sqrt{1+|u'|^2}$ is constant while $|u'(x)|\to\infty$ as $|x|\to1$.
  We finally remark that when $p=2$, the graph curve of $u$ corresponds to the so-called rectangular elastica, and our closed curve coincides with the one constructed by Henrot-Mounjid.
\end{remark}

Now we turn to the proof of Theorem \ref{thm:diameter-degeneracy}.

\begin{proof}[Proof of Theorem \ref{thm:diameter-degeneracy}]
  Up to rescaling we may assume that $\di(\Sigma)=1$ and only need to prove that
  \begin{equation}\label{eq01}
    \int_\Sigma H^p > c_p \mathcal{D}^{p-1}.
  \end{equation}

  {\em Step 1.}
  Choose one direction $\omega\in\mathbb{S}^2$ such that $b_\Sigma(\omega)=1$ ($=\di(\Sigma)$).
  Up to a rigid motion, we may assume that the height function $h(q) := q\cdot\omega$ maps $\Sigma$ to $[0,1]$.
  Let $\Sigma_t$ denote the cross section $\{q\in\Sigma \mid h(q)=t\}$ for $t\in(0,1)$.
  In addition, let $\theta_\omega\in[0,\pi]$ denote the angle between $\omega$ and the outer unit normal $\nu$ of $\Sigma$, so that $\cos\theta_\omega=\nu\cdot\omega$.
  Note that $\sin\theta_\omega>0$ for $t\in(0,1)$.
  Then the co-area formula yields
  \begin{equation}\label{eq02}
    \int_\Sigma H^p d\mathcal{H}^2 = \int_0^1 \int_{\Sigma_t} \frac{H^p}{\sin\theta_\omega} d\mathcal{H}^1 dt,
  \end{equation}
  where $\mathcal{H}^d$ denotes the $d$-dimensional Hausdorff measure.
  Moreover, at any point $q\in\Sigma$ such that $t=h(q)\in(0,1)$, let $k_{\Sigma_t}$ be the inward curvature of the cross section curve $\Sigma_t$, and let $\kappa_\omega$ be the inward curvature of a (unique) curve contained in $\Sigma$ and the plane $P:=q+\mathrm{span}\{\nu(q),\omega\}$.
  Then from a simple geometric calculation we deduce
  \begin{equation}\label{eq37}
    2H = k_{\Sigma_t}\sin\theta_\omega + \kappa_\omega.
  \end{equation}
  Therefore, inserting \eqref{eq37} into \eqref{eq02}, and using the fact that $(X+Y)^p \geq X^p$ for $X,Y\geq0$ with equality only for $Y=0$, we obtain
  \begin{align}\label{eq03}
    \int_\Sigma H^p d\mathcal{H}^2 > \frac{1}{2^p} \int_0^1 \int_{\Sigma_t} k_{\Sigma_t}^p(\sin\theta_\omega)^{p-1} d\mathcal{H}^1 dt.
  \end{align}
  The strict positivity follows since otherwise $\kappa_\omega\equiv0$ but this contradicts the fact that $\Sigma$ is closed e.g.\ in view of the Gauss-Bonnet theorem.

  {\em Step 2.}
  We then prove that for every $t\in(0,1)$ and $q\in\Sigma$ with $t=h(q)$,
  \begin{equation}\label{eq05}
    \sin\theta_\omega(q) \geq g(t) := \left(1+\left(\tfrac{1}{2}-|t-\tfrac{1}{2}|\right)^{-2}\right)^{-\frac{1}{2}}.
  \end{equation}
  By symmetry we only need to argue for $t\leq\frac{1}{2}$ and prove that
  \begin{equation}\label{eq06}
    \sin\theta_\omega(q) \geq \left(1+t^{-2}\right)^{-1/2}.
  \end{equation}
  Fix $q\in\Sigma$ (and hence also $t=h(q)\leq\frac{1}{2}$).
  Up to a rigid motion, we may assume that the maximum (resp.\ minimum) of the height function $h$ is attained by $(1,0,0)$ (resp.\ $(0,0,0)$), so that $\omega=(1,0,0)$ in particular, and also that there is some function
  \begin{align}\label{eq07}
    f:[0,1]\to[0,1] \ \mbox{concave}, \quad f(0)=f(1)=0, \quad
  \end{align}
  such that $q \in G_f \subset \Sigma$, where $G_f:=\{(x,0,f(x))\in\mathbb{R}^3 \mid x\in[0,1]\}$.
  Note that the upper bound $f\leq1$ follows since $\di(\Sigma)=1$.
  Then an elementary geometry implies that $\sin\theta_\omega\geq\sin\theta_\omega'$ for the angle $\theta_\omega'$ between $\omega$ and the normal $(-f'(t),0,1)$ of $f$, that is,
  \begin{equation}\label{eq08}
    \sin\theta_\omega(q) \geq \frac{1}{\sqrt{1+|f'(t)|^2}}.
  \end{equation}
  In addition, by the geometric restriction \eqref{eq07} of $f$, we have for $t\leq\frac{1}{2}$,
  \begin{equation}\label{eq09}
    |f'(t)|\leq \frac{1}{t}.
  \end{equation}
  Indeed, $f'(t)\leq1/t$ holds since otherwise $f'(x)>1/t$ for $x\in(0,t)$ by concavity but this contradicts $f\leq1$ and $f(0)=0$; by symmetry, using $f(1)=0$, we also obtain $f'(t)\geq 1/(1-t)$; since $t\leq\frac{1}{2}$, these two estimates imply \eqref{eq09}.
  From \eqref{eq08} and \eqref{eq09} we deduce the desired \eqref{eq06}.

  {\em Step 3.}
  We finally complete the proof.
  Inserting \eqref{eq05} into \eqref{eq03}, we now obtain
  \begin{align}\label{eq38}
    \int_\Sigma H^p d\mathcal{H}^2 > \frac{1}{2^p} \int_0^1 g(t)^{p-1} \int_{\Sigma_t} k_{\Sigma_t}^p d\mathcal{H}^1 dt.
  \end{align}
  By definition of $\mathcal{D}$ (and $\di(\Sigma)=1$), we can apply Lemma \ref{lem:degeneracy} with $r=1/\mathcal{D}$ to $\Sigma_t$ to the effect that
  \begin{align}\label{eq39}
    \int_{\Sigma_t} k_{\Sigma_t}^p d\mathcal{H}^1 \geq \tilde{c}_p \mathcal{D}^{p-1},
  \end{align}
  where the right-hand side does not depend on $t$.
  Therefore, inserting \eqref{eq39} to \eqref{eq38}, we obtain the desired \eqref{eq01} for $c_p:=2^{-p}\tilde{c}_p\big(\int_0^1g(t)^{p-1}dt\big)$; this constant agrees with the one in the statement of Theorem \ref{thm:diameter-degeneracy} after simple calculations.
\end{proof}

We now estimate the degeneracy $\mathcal{D}$ to prove Theorem \ref{thm:main1}.
A key fact we use is the following scaling law (whose prefactor is not optimal).

\begin{lemma}
  For a convex surface $\Sigma$, we have
  \begin{equation}
    \di(\Sigma)\frac{A}{V} \leq 36 \mathcal{D}.
  \end{equation}
\end{lemma}

\begin{proof}
  Up to rescaling we may assume that $\di(\Sigma)=1$.
  Let $r:=1/\mathcal{D}$ for notational simplicity.
  Fixing a diameter direction $\omega_0\in\mathbb{S}^2$ such that $b_\Sigma(\omega_0)=1$, we let $\rho_1\in[r,1]$ denote the maximal width among all directions orthogonal to $\omega_0$, that is, $\rho_1:=\max\{b_\Sigma(\omega)\mid \omega\cdot\omega_0=0\}$, and $\omega_1\in\mathbb{S}^2$ be a maximizer so that $b_\Sigma(\omega_1)=\rho_1$ and $\omega_1\cdot\omega_0=0$.
  We now separately prove
  \begin{equation}\label{eq12}
    A(\Sigma) \leq 6\rho_1 \quad \mbox{and} \quad V(\Sigma) \geq \frac{r\rho_1}{6},
  \end{equation}
  which immediately imply $A/V \leq 36/r = 36 \mathcal{D}$.

  To this end we use the fact that there exists a rectangular (convex body) $Q\subset\mathbb{R}^3$ containing $\Sigma$ such that each side is perpendicular to one of $\omega_0,\omega_1,\omega_2$, where $\omega_2\in\mathbb{S}^2$ is chosen to be orthogonal to both $\omega_0$ and $\omega_1$, and such that the side-lengths of $Q$ are $1,\rho_1,\rho_2$, where $\rho_2:=b_\Sigma(\omega_2)\in[r,\rho_1]$.
  Then the first estimate in \eqref{eq12} follows by the area-monotonicity of convex surfaces that $\Sigma\subset Q \Rightarrow A(\Sigma)\leq A(\partial Q)$, which combined with $\rho_2\leq\rho_1\leq1$ implies that $A(\Sigma)\leq A(\partial Q)=2(\rho_2+\rho_1+\rho_1\rho_2)\leq 6\rho_1$.
  For the second estimate in \eqref{eq12} we further use the fact that $\Sigma$ touches all sides of $Q$.
  More precisely, assuming without loss of generality that $\omega_0,\omega_1,\omega_2$ form the standard basis of $\mathbb{R}^3$ and that $Q=[0,1]\times[0,\rho_1]\times[0,\rho_2]$, we can find points in $\Sigma\cap\partial Q$ of the form
  \begin{equation}\label{eq41}
    (0,a_1,a_2),(1,a_1',a_2'),(a_3,0,a_4),(a_3',\rho,a_4'),(a_5,a_6,0),(a_5',a_6',r)\in\Sigma\cap\partial Q.
  \end{equation}
  In addition, since $\omega_0$ and $\omega_1$ are defined via maximization, we have
  \begin{equation}\label{eq42}
    a_i=a_i' \quad \mbox{for}\ i=1,2,4.
  \end{equation}
  Then the polyhedron $P$ defined by the convex hull of the points in \eqref{eq41} is enclosed by $\Sigma$, and in addition under the constraint \eqref{eq42} we deduce from a direct computation that the enclosed volume of $P$ is $\rho_2\rho_1/6$, so that $V(\Sigma)\geq \rho_2\rho_1/6 \geq r\rho_1/6$.
  The proof is complete.
\end{proof}

We are now in a position to complete the proof of Theorem \ref{thm:main1}.

\begin{proof}[Proof of Theorem \ref{thm:main1}]
  By Theorem \ref{thm:diameter-degeneracy} with $p=2$ and Lemma \ref{lem:degeneracy}, we obtain
  $$\di(\Sigma)\frac{A}{V} \leq \frac{36}{c_2} \int_\Sigma H^2.$$
  In addition, since the total mean curvature of a convex surface can be represented by the mean width, namely $\int_\Sigma H = 2\pi B$, where $B = \frac{1}{|\mathbb{S}^2|}\int_{\mathbb{S}^2}b_\Sigma(\nu)dS(\nu)$ (cf.\ (19) \& (22) in \cite[Chapter~4]{Burago1988}), and since $b_\Sigma(\nu)\leq \di(\Sigma)$ in every direction $\nu$ by definition of diameter, we have
  \begin{equation}\label{eq35}
    \int_{\Sigma}H \leq 2\pi\di(\Sigma),
  \end{equation}
  completing the proof with $C=72\pi/c_2$.
  (As $c_2\geq 2/3$, we can take $C=108\pi$.)
\end{proof}

In the rest of this section we briefly observe that $\mathcal{D}$ can be also related with other scale invariant quantities with optimal exponents; the reader may skip this part as these estimates are not used any other part of this paper.
For notational simplicity, we let $d:=\di(\Sigma)$ and $M:=\int_\Sigma H$.

We begin with indicating that in fact the converse of Lemma \ref{lem:degeneracy} also holds, i.e.,
\begin{equation}\label{eq31}
  \mathcal{D} \sim d\frac{A}{V} \sim M\frac{A}{V}.
\end{equation}
Indeed, again letting $d=1$, $r=1/\mathcal{D}$, and $\omega_0$ be such that $b_\Sigma(\omega_0)=1$, if we choose $\omega_2$ to be attaining the minimal width so that $b_{\Sigma}(\omega_2)=r=:\rho_2$, and $\omega_1$ to be orthogonal to both $\omega_0$ and $\omega_2$, and write $\rho_1=b_{\Sigma}(\omega_1)$, and in addition if we similarly take a rectangular $Q$ and a polyhedron $P$ to the proof of Lemma \ref{lem:degeneracy}, then we have $A(\Sigma)\geq A(\partial P) \gtrsim \rho_1$ and $V(\Sigma)\leq V(\partial Q) \lesssim \rho_1\rho_2=\rho_1r$ so that $A/V \gtrsim 1/r$.
Therefore, after retrieving $d$ and combining with Lemma \ref{lem:degeneracy}, we obtain the first relation in \eqref{eq31}.
The second one follows by the fact that $d\sim M$ holds under convexity; in fact, we have $d \lesssim M$ (even without convexity by Topping's inequality \cite{Topping2008}), while \eqref{eq35} implies that $M \lesssim d$ under convexity.

Next we focus on the isoperimetric ratio $I=A^\frac{3}{2}/V$, for which we have
\begin{equation}\label{eq32}
  I \lesssim \mathcal{D} \lesssim I^2.
\end{equation}
The first one is already observed, cf.\ \eqref{eq31} and \eqref{eqMinkowski}, while the second one follows since $3VM\leq A^2$ (cf.\ \cite[p.145]{Burago1988}).
Note that both sides in \eqref{eq32} are optimal because for a pancake-like (resp.\ cigar-like) surface $\Sigma_\varepsilon$, we have $I=O(\varepsilon^{-1})\sim\mathcal{D}$ (resp.\ $I^2=O(\varepsilon^{-1})\sim\mathcal{D}$).

We finally indicate that the ratio $R:=r_\mathrm{out}/r_\mathrm{in}$, where $r_\mathrm{out}$ and $r_\mathrm{in}$ are the circumradius and the inradius, respectively, is completely comparable with $\mathcal{D}$:
\begin{equation}\label{eq36}
  R \sim \mathcal{D}.
\end{equation}
Indeed, assuming that $d=1$ up to rescaling, we obviously have $r_\mathrm{out}\sim 1$ and $r_\mathrm{in}\leq 1/\mathcal{D}$, and hence $\mathcal{D} \lesssim R$; in addition, since the polyhedron $P$ in the proof of Lemma \ref{lem:degeneracy} encloses a ball of radius $\sim 1/\mathcal{D}$, 
we also have $\mathcal{D} \gtrsim 1/r_\mathrm{in}\sim R$.

\section{Topping's conjecture for axisymmetric surfaces}\label{sec:Topping}

In this section we prove Theorem \ref{thm:Topping}, namely Topping's conjecture with a rigidity estimate.
Throughout this section, for notational simplicity, we let
$$M(\Sigma):=\int_\Sigma|H|.$$
Since we focus on axisymmetric surfaces, we may hereafter assume the following

\begin{hypothesis}[Simply-connected axisymmetric surface]\label{hyp:axisymmetric}
  A surface $\Sigma$ is represented by using an immersed $C^{1,1}$ plane curve $\gamma=(x,z):[0,L]\to\mathbb{R}^2$ as
  $$\Sigma = \left\{ \big( x(s)\cos\phi,x(s)\sin\phi,z(s) \big) \in\mathbb{R}^3 \mid 0\leq s\leq L,\ 0\leq\phi<2\pi \right\},$$
  where $\gamma$ is parametrized by the arclength $s\in[0,L]$, and satisfies $x(0)=z(0)=x(L)=0$, $z(L)\geq0$,
  $\gamma_s(0)=(1,0)$, $\gamma_s(L)=(-1,0)$, and also $x(s)>0$ for any $s\in(0,L)$.
  (The notation $\gamma_s$ means the derivative with respect to $s$.)
  For such a surface, throughout this section, we let $\theta:[0,L]\to\mathbb{R}$ denote a unique Lipschitz function such that for $s\in[0,L]$,
  \begin{equation*}
    x(s)=\int_0^s\cos\theta(t)dt, \quad z(s)=\int_0^s\sin\theta(t)dt,
  \end{equation*}
  and such that $\theta(0)=0$.
  Note that $\theta(L)\in\pi+2\pi\mathbb{Z}$.
\end{hypothesis}

Our strategy is to reduce the problem into the convex one by constructing, for a given $\Sigma$ satisfying Hypothesis \ref{hyp:axisymmetric}, a comparison axisymmetric convex surface $\Sigma'$ such that $M(\Sigma)\geq M(\Sigma')$ and $\di(\Sigma)\leq\di(\Sigma')$.
To this end, following the strategy in \cite{Dalphin2016}, we perform two kinds of rearrangement.
The first one is to make the curve not going down vertically, while keeping the horizontal behavior.
The second one is to make the curve convex by rearranging the tangential angle.
The reason why we weaken the regularity of $\Sigma$ to $C^{1,1}$ in Hypothesis \ref{hyp:axisymmetric} is that the best regularity retained in these rearrangements is the Lipschitz continuity of $\theta$, that is the $C^{1,1}$-regularity of $\Sigma$.
(Notice that the $C^{1,1}$-regularity is enough for defining $M$ since $C^{1,1}=W^{2,\infty}$.)

\subsection{First rearrangement}\label{subsec:first}

Given $\theta$ as in Hypothesis \ref{hyp:axisymmetric}, we define $\theta^\sharp$ by
\begin{equation}\label{eq:firstrearrangement}
  \theta^\sharp(s):=\mathrm{dist}(\theta(s),2\pi\mathbb{{Z}}) \quad \mbox{for}\ s\in[0,L],
\end{equation}
so that $\theta^\sharp([0,L])=[0,\pi]$.
Note that $\theta^\sharp$ is Lipschitz, $\theta^\sharp(0)=0$ and $\theta^\sharp(L)=\pi$.
For the corresponding curve $\gamma_\sharp=(x_\sharp,z_\sharp)$ starting from the origin, we have $x_\sharp(L)=0$ as $\cos\theta^\sharp=\cos\theta$, and $z_\sharp(L)>0$ as $\sin\theta^\sharp\geq0$ and $\sin\theta^\sharp\not\equiv0$.
Hence the corresponding surface $\Sigma_\sharp$ still satisfies Hypothesis \ref{hyp:axisymmetric}.

For the first rearrangement we have
\begin{equation}\label{eq:comparison1}
  M(\Sigma)=M(\Sigma_\sharp), \quad \di(\Sigma)\leq\di(\Sigma_\sharp).
\end{equation}
Since the equality for $M$ is already known (cf. \cite{Dalphin2016} or Appendix \ref{app:meancurvature}), we only need to check the diameter control, which is not difficult.

\begin{lemma}[Diameter control: First rearrangement]\label{lem:diameterfirst}
  Let $\Sigma$ satisfy Hypothesis \ref{hyp:axisymmetric}, and let $\Sigma_\sharp$ be obtained by the first rearrangement of $\Sigma$.
  Then $\di(\Sigma)\leq\di(\Sigma_\sharp)$.
\end{lemma}

\begin{proof}
  Notice the general fact for an axisymmetric surface that, using the reflection operator $R:(x,z)\mapsto(-x,z)$, we can represent $\di(\Sigma)$ by the maximum of $\mathrm{dist}(\gamma(s_1),R\gamma(s_2))$ over $s_1,s_2\in[0,L]$.
  Notice that $x\equiv x_\sharp$ holds since $\cos\theta\equiv\cos\theta^\sharp$ follows by definition of the first rearrangement.
  Therefore, it now suffices to show that the vertical distance of any two points does not contract, i.e.,
  \begin{equation*}
    |z_\sharp(s_2)-z_\sharp(s_1)| \geq |z(s_2)-z(s_1)| \quad \mbox{for}\ 0\leq s_1<s_2 \leq L.
  \end{equation*}
  This follows by the fact that $\sin\theta^\sharp=|\sin\theta|$ so that
  $$|z_\sharp(s_2)-z_\sharp(s_1)|=\big|\int_{s_1}^{s_2}\sin\theta^\sharp\big|=\int_{s_1}^{s_2}|\sin\theta| \geq \big|\int_{s_1}^{s_2}\sin\theta\big|=|z(s_2)-z(s_1)|.$$
  The proof is complete.
\end{proof}

\subsection{Second rearrangement}\label{subsec:second}

Given $\theta$ as in Hypothesis \ref{hyp:axisymmetric} such that $\theta([0,L])=[0,\pi]$, we define $\theta^*:[0,L]\to[0,\pi]$ by the standard nondecreasing rearrangement of $\theta$:
\begin{equation}\label{eq:secondrearrangement}
  \theta^*(s) = \sup\{ c \in[0,\pi] \mid s \geq L-m(\{\theta\geq c\}) \} \quad \mbox{for}\ s\in[0,L],
\end{equation}
where $m(\{\theta\geq c\})$ means the Lebesgue measure of $\{s\in[0,L] \mid \theta(s)\geq c\}$ (see e.g.\ \cite{Kawohl1985} for details of the rearrangement argument).
Then the resulting surface $\Sigma_*$ is clearly convex as $\theta^*$ is monotone.
In addition, $\Sigma_*$ still satisfies Hypothesis \ref{hyp:axisymmetric}; indeed, thanks to well-known properties of the rearrangement, the function $\theta^*$ inherits the Lipschitz continuity of $\theta$ \cite[Lemma 2.3]{Kawohl1985}, and the corresponding curve $\gamma_*=(x_*,z_*)$ starting from the origin retains all the boundary conditions, i.e.,
\begin{equation}
  \theta^*(0)=\theta(0)=0,\quad \theta^*(L)=\theta(L)=\pi, \quad \gamma_*(0)=\gamma(0),\quad \gamma_*(L)=\gamma(L),
\end{equation}
where in particular the last condition follows by the integration-preserving property: $\int_0^Lf(\theta(s))ds=\int_0^Lf(\theta^*(s))ds$ for any continuous function $f$ \cite[p.22, (C)]{Kawohl1985}.

For the second rearrangement we have
\begin{equation}\label{eq:comparison2}
  M(\Sigma)\geq M(\Sigma_*), \quad \di(\Sigma)\leq\di(\Sigma_*).
\end{equation}
Here the remaining task is again only to establish the diameter control (cf.\ \cite{Dalphin2016} or Appendix \ref{app:meancurvature}).
This second diameter control needs a more delicate argument, but it turns out that the following fine property holds.

\begin{lemma}[Diameter control: Second rearrangement]\label{lem:diameter}
  Let $\Sigma$ satisfy Hypothesis \ref{hyp:axisymmetric}.
  Suppose that $\theta([0,L])=[0,\pi]$.
  Let $\Sigma_*$ be the convex surface obtained by the second rearrangement of $\Sigma$.
  Then $\Sigma_*$ encloses $\Sigma$.
  In particular, $\di(\Sigma_*)\geq\di(\Sigma)$.
\end{lemma}

\begin{proof}
  {\em Step 1.}
  We first reduce the problem by using symmetry.
  Fix a unique point $\bar{s}\in[0,L]$ such that $\theta^*(\bar{s})=\pi/2$ and such that $\theta^*(s)<\pi/2$ for every $s<\bar{s}$; in other words, $\bar{s}$ is the first point where $x_*$ attains the maximum.
  Then we can represent the convex curve $\gamma_*$ (corresponding to $\Sigma_*$) on the restricted interval $[0,\bar{s}]$ by a graph curve, namely,
  \begin{equation}\label{eq19}
    \begin{split}
      \mbox{there is a nondecreasing convex function}\ U_* \ \mbox{such that}\\
      \gamma_*([0,\bar{s}]) = G_* :=\{(x,U_*(x))\in\mathbb{R}^2 \mid x\in[0,x_*(\bar{s})]\}.
    \end{split}
  \end{equation}
  Notice that by symmetry it is sufficient for Lemma \ref{lem:diameter} to prove that the image of $\gamma$ is included in the epigraph of $U_*$:
  \begin{equation}\label{eq20}
    \gamma([0,L]) \subset G_*^+:=\{(x,z)\in\mathbb{R}^2 \mid z\geq U_*(x),\ x\in[0,x_*(\bar{s})] \}.
  \end{equation}
  Indeed, if we establish \eqref{eq20}, then since the procedures of rearrangement and vertical reflection are commutative, using \eqref{eq20} also for the reflected curve $\tilde{\gamma}(s)=(x(L-s),z(L)-z(L-s))$, we find that the same kind of inclusion as \eqref{eq20} also holds for the subgraph of the upper-part of $\gamma_*$, so that the desired assertion holds.

  For later use we put down an elementary geometric property of $G_*^+$, cf.\ \eqref{eq19}:
  \begin{equation}\label{eq22}
    (x,z)\in G_*^+ \quad \Longrightarrow \quad \Big( 0\leq x'\leq x,\ z'\geq z\ \Rightarrow (x',z')\in G_*^+ \Big).
  \end{equation}

  {\em Step 2.}
  Now we prove \eqref{eq20}; more precisely, we fix an arbitrary $s_\dagger\in[0,L]$ and prove that $(x(s_\dagger),z(s_\dagger))\in G_*^+$.
  Let $\theta^\dagger:=(\theta|_{[0,s_\dagger]})^*$, i.e., the nondecreasing rearrangement of the restriction $\theta|_{[0,s_\dagger]}$.
  Let $\gamma_\dagger=(x_\dagger,z_\dagger)$ be the corresponding convex curve defined on $[0,s_\dagger]$, which in particular satisfies
  \begin{equation}\label{eq23}
    \gamma_\dagger(0)=(0,0), \quad \gamma_\dagger(s_\dagger)=\gamma(s_\dagger),
  \end{equation}
  by the integration-preserving property of rearrangement.
  Thus we only need to prove that $\gamma_\dagger(s_\dagger)\in G_*^+$.
  In what follows we prove the stronger assertion that
  \begin{equation}\label{eq55}
    \gamma_\dagger([0,s_\dagger])\subset G_*^+.
  \end{equation}

  We first notice that as the general property of rearrangement,
  \begin{equation}\label{eq15}
    \theta^\dagger(s) \geq \theta^*(s) \quad \mbox{for every} \ s\in[0,s_\dagger].
  \end{equation}
  Indeed, letting $\varphi:=\theta\chi_{[0,s_\dagger]}+\pi\chi_{(s_\dagger,L]}$, where $\chi$ denotes the characteristic function, we have $\varphi\geq\theta$ on $[0,L]$ and hence $\varphi^*\geq\theta^*$ by the order-preserving property \cite[p.21, (M1)]{Kawohl1985}, and also $\theta^\dagger(s)=\varphi^*(s)$ for $s\in[0,s_\dagger]$.

  Using \eqref{eq15}, we prove that
  \begin{equation}\label{eq16}
    0 \leq x_\dagger(s) \leq x_*(s) \quad \mbox{for every} \ s\in[0,s_\dagger].
  \end{equation}
  Indeed, since $\cos\theta$ is decreasing on $[0,\pi]$, and since $\theta^\dagger$ is nondecreasing on $[0,s_\dagger]$ and valued into $[0,\pi]$, the function $x_\dagger(s)=\int_0^s\cos\theta^\dagger$ is concave on $[0,s_\dagger]$ and hence $x_\dagger(s)\geq\min\{x_\dagger(0),x_\dagger(s_\dagger)\}\geq0$, cf.\ \eqref{eq23}.
  In addition, thanks to \eqref{eq15}, we have $x_\dagger(s)=\int_0^s\cos\theta^\dagger\leq\int_0^s\cos\theta^*=x_*(s)$, completing the proof of \eqref{eq16}.

  We now prove \eqref{eq55} by considering the behaviors of $z_*$ and $z_\dagger$.
  Below we separately consider the two intervals $[0,\sigma_\dagger]$ and $[\sigma_\dagger,s_\dagger]$, where $\sigma_\dagger\in[0,s_\dagger]$ denotes the maximal $s\in[0,s_\dagger]$ such that $\theta_\dagger(s)\leq\pi/2$.
  (Note that $\sigma_\dagger$ may coincide with $s_\dagger$.)
  Concerning the first interval $[0,\sigma_\dagger]$, we have
  \begin{equation}\label{eq17}
    z_\dagger(s) \geq z_*(s) \quad \mbox{for every} \ s\in[0,\sigma_\dagger].
  \end{equation}
  Indeed, if $s\leq\sigma_\dagger$, then $\theta^*([0,s])\subset\theta^\dagger([0,s])\subset[0,\pi/2]$ by \eqref{eq15} and by definition of $\sigma_\dagger$; hence, by \eqref{eq15} and by the fact that $\sin\theta$ is increasing on $[0,\pi/2]$, we obtain $z_\dagger(s)=\int_0^s\sin\theta^\dagger\geq\int_0^s\sin\theta^*=z_*(s)$, completing the proof of \eqref{eq17}.
  Therefore, by using \eqref{eq16}, \eqref{eq17}, and the obvious inclusion $\gamma_*([0,\sigma_\dagger]) \subset G_*^+$, we deduce from the geometric property \eqref{eq22} that
  \begin{equation}\label{eq24}
    \gamma_\dagger([0,\sigma_\dagger])\subset G_*^+.
  \end{equation}
  Concerning the remaining part $[\sigma_\dagger,s_\dagger]$, since $\theta^\dagger([\sigma_\dagger,s_\dagger])\subset[\pi/2,\pi]$ by definition of $\sigma_\dagger$, we have $(x_\dagger)_s=\cos\theta^\dagger\leq0$ and $(z_\dagger)_s=\sin\theta^\dagger\geq0$ on $[\sigma_\dagger,s_\dagger]$, and hence
  $$x_\dagger(s)\leq x_\dagger(\sigma_\dagger), \quad z_\dagger(s)\geq z_\dagger(\sigma_\dagger) \quad \mbox{for every} \ s\in[\sigma_\dagger,s_\dagger].$$
  Using this property with the facts that $x_\dagger\geq0$, cf.\ \eqref{eq16}, and that $(x_\dagger(\sigma_\dagger),z_\dagger(\sigma_\dagger))\in G_*^+$, cf.\ \eqref{eq24}, we deduce from the geometric property \eqref{eq22} that
  $\gamma_\dagger([\sigma_\dagger,s_\dagger])\subset G_*^+.$
  This combined with \eqref{eq24} implies \eqref{eq55}, thus completing the proof.
\end{proof}

\subsection{Rigidity estimates}\label{subsec:rigidity}

Topping's conjecture itself is now already proved for simply-connected axisymmetric surfaces by the above two subsections.
In this final subsection we establish Theorem \ref{thm:Topping} by giving a more quantitative analysis of the rearrangements.
More precisely, we prove that the deficit $\frac{M(\Sigma)}{\di(\Sigma)}-\pi$ is bounded below by the sum of quantities measuring ``axial expansion in the first rearrangement'' and ``coaxial deviation after the second rearrangement''.

We first prepare Lemma \ref{lem:convexmeanwidth} below about ``coaxial deviation'' in the framework of convex geometry.
To this end we briefly recall some classical facts in convex geometry.
Since $M$ has the mean-width representation, $M$ is naturally defined even for singular (Lipschitz) convex surfaces, and moreover $M$ has strict monotonicity with respect to the inclusion property for enclosed convex sets.
In addition, we can also define $M$ even for a degenerate convex body $K$ (of dimension $\leq2$) by $M(K):=\lim_{\varepsilon\to0}M(\partial K_\varepsilon)$, where $K_\varepsilon:=\{x\in\mathbb{R}^3 \mid \mathrm{dist}(x,K)\leq\varepsilon\}$ denotes the $\varepsilon$-neighborhood of $K$.
For example, for a segment $S$ of length $\ell$, we have $M(S)=\pi\ell$.
Notice that the above monotonicity is valid even for degenerate objects.
These facts in particular imply \eqref{eqTopping} for any convex surface $\Sigma$; indeed, if we take a segment $S$ attaining the diameter of $\Sigma$, then from the monotonicity of $M$ and the fact that $\di(\Sigma)=\di(S)$ we deduce that
$$\frac{M(\Sigma)}{\di(\Sigma)}>\frac{M(S)}{\di(S)}=\pi.$$
Therefore, in order to obtain a lower bound for the deficit $\frac{M(\Sigma)}{\di(\Sigma)}-\pi$, it is natural to look at the quantity $M(\Sigma)-M(S)$.

We are now ready to rigorously state a lemma concerning coaxial deviation.

\begin{lemma}[Coaxial deviation]\label{lem:convexmeanwidth}
  Let $\Sigma$ be an axisymmetric convex surface.
  Then
  \begin{equation}\label{eq50}
    M(\Sigma)-M(S) \gtrsim \frac{b^2}{\di(\Sigma)},
  \end{equation}
  where $b$ denotes the maximal distance from an axis of symmetry $L_\mathrm{axis}$, i.e., $b:=\max_{q\in\Sigma}\mathrm{dist}(q,L_\mathrm{axis})$, and $S$ denotes a segment attaining $\di(\Sigma)$.
\end{lemma}

Before entering the proof, we compute the energy $M$ for a useful example of a (singular) convex surface, which plays an important role in the proof of Lemma \ref{lem:convexmeanwidth} as a comparison surface.

\begin{remark}[A useful example $\Gamma^h_{a,A}$]\label{rem:examplecomputation}
  Given $h>0$ and $0< a\leq A$, we let $\Gamma^h_{a,A}$ denote the surface defined by the boundary of the convex hull of the two circles $C_\pm:=\{x^2+y^2=a^2,\ |z|=\pm h/2\}$ and the additional intermediate circle $C_0:=\{x^2+y^2=A^2,\ z=0\}$.
  We compute
  \begin{align}\label{eq52}
    & M(\Gamma^h_{a,A}) = \pi h + \pi^2a + 2\pi\theta(A-a),\\
    & \quad \mbox{where $\theta\in[0,\pi/2]$ such that $\tan\theta=2(A-a)/h$.} \nonumber
  \end{align}
  Indeed, an explicit calculation shows that the smooth (conical) parts $\Gamma^h_{a,A}\cap\{0<|z|<h/2\}$ has the energy $\pi h$, not depending on $a$ nor $A$.
  The energy on the singular parts $C_\pm$ and $C_0$ only depend on the angles and the lengths of the edges, in view of approximation by the $\varepsilon$-neighborhood.
  Since the above $\theta$ denotes the angle between the $z$-axis and a generating line of the conical part, the deviation angle on $C_\pm$ is $\pi/2-\theta$, and that on $C_0$ is $2\theta$.
  Then we compute the energy on $C_+$ (or $C_-$) is $\frac{1}{2}\cdot(\pi/2-\theta)\cdot2\pi a$, and that on $C_0$ is $\frac{1}{2}\cdot2\theta\cdot2\pi A$.
  Summing up all implies \eqref{eq52}.
  Note that the above computation is valid not only for a cylinder ($0<a=A$) but also for degenerate cases: e.g.\ double-cone ($0=a<A$), segment ($a=A=0$), and disk ($h=0$, and hence $\theta=\pi/2$).
\end{remark}

Now we turn to the proof of Lemma \ref{lem:convexmeanwidth}.

\begin{proof}[Proof of Lemma \ref{lem:convexmeanwidth}]
  Let $r$ be the maximal distance from $L_{\mathrm{axis}}$ concerning $S$, i.e., $r:=\max_{p\in S}\mathrm{dist}(p,L_\mathrm{axis})$.
  Since $S$ is of length $d:=\di(\Sigma)$, up to a rigid motion, the surface $\Sigma$ encloses the cylinder $\Gamma^h_{r,r}$ of radius $r$ and height $h:=\sqrt{d^2-4r^2}$.
  Recall that $M(\Gamma^h_{r,r}) = \pi (h+\pi r)$, cf.\ Remark \ref{rem:examplecomputation}.
  Using the monotonicity $M(\Sigma)\geq M(\Gamma^h_{r,r})\geq M(S)=\pi d$, and noting that $h-(d-2r)\geq0$, we have
  $$M(\Sigma)-M(S) \geq M(\Gamma^h_{r,r})-M(S) = \pi (h+\pi r) - \pi d \geq \pi(\pi-2)r.$$
  In the case of $r\not\ll b$, say $r\geq b/4$, this implies \eqref{eq50} with the help of the obvious estimate $b\leq d$.

  We now consider the case of $r\leq b/4$.
  By convexity of $\Sigma$ and definition of $b$, the surface $\Sigma$ also encloses $\Gamma^h_{r,b/2}$.
  Then from monotonicity, Remark \ref{rem:examplecomputation}, and the assumption $r\leq b/4$, we deduce that
  \begin{equation}\label{eq51}
    M(\Sigma)-M(S)\geq M(\Gamma^h_{r,b/2})-M(\Gamma^h_{r,r}) = 2\pi\theta(b/2-r) \gtrsim \theta b,
  \end{equation}
  where $\theta\in[0,\pi/2]$ satisfies that $\tan\theta=(b-2r)/h\gtrsim b/h$.
  If $\theta\geq\pi/4$, then \eqref{eq51} and $b\leq d$ again directly imply \eqref{eq50}.
  If $\theta\leq\pi/4$, then we can use the estimate $\theta\gtrsim\tan\theta$ for \eqref{eq51} to the effect that
  $$M(\Sigma)-M(S) \gtrsim b\tan\theta \gtrsim b^2/h.$$
  Then the obvious estimate $h\leq d$ implies \eqref{eq50}, completing the proof.
\end{proof}

With Lemma \ref{lem:convexmeanwidth} at hand, we are now in a position to prove Theorem \ref{thm:Topping}.

\begin{proof}[Proof of Theorem \ref{thm:Topping}]
  We may assume Hypothesis \ref{hyp:axisymmetric} on $\Sigma$.
  Throughout the proof, we let $\Sigma_\sharp$ denote the surface given by the first rearrangement of $\Sigma$, and $\Sigma_*$ by the second rearrangement of $\Sigma_\sharp$.
  For notational simplicity we let $d$ (resp.\ $d_\sharp$, $d_*$) denote $\di(\Sigma)$ (resp.\ $\di(\Sigma_\sharp)$, $\di(\Sigma_*)$).
  Notice that $d\leq d_\sharp\leq d_*\leq L$, cf.\ Lemmata \ref{lem:diameterfirst} and \ref{lem:diameter}, where $L$ denotes the (same) length of generating curves of those three surfaces.
  We divide the proof into three steps.

  {\em Step 1.}
  We first prove that
  \begin{equation}\label{eq45}
    \frac{M(\Sigma)}{d} - \pi \gtrsim \frac{b_*^2}{L^2} + \frac{d_*-d}{L}.
  \end{equation}
  Using \eqref{eq:comparison1} and $d\leq L$, we obtain
  \begin{equation}\label{eq43}
    \frac{M(\Sigma)}{d} = \frac{M(\Sigma)}{d_\sharp}\left(1+\frac{d_\sharp-d}{d}\right) \geq \frac{M(\Sigma_\sharp)}{d_\sharp}\left(1+\frac{d_\sharp-d}{L}\right).
  \end{equation}
  Similarly, using \eqref{eq:comparison2} and $d_\sharp\leq L$, we get
  \begin{equation}\label{eq57}
    \frac{M(\Sigma_\sharp)}{d_\sharp} \geq \frac{M(\Sigma_*)}{d_*}\left(1+\frac{d_*-d_\sharp}{L}\right).
  \end{equation}
  Then, using Lemma \ref{lem:convexmeanwidth} for a segment $S_*$ attaining the diameter $d_*$ of $\Sigma_*$, and also using $d_*\leq L$, we deduce that there is a universal constant $\sigma>0$ such that
  \begin{equation}\label{eq44}
    \frac{M(\Sigma_*)}{d_*} = \pi + \frac{M(\Sigma_*)-M(S_*)}{d_*} \geq \pi + \sigma\frac{b_*^2}{d_*^2} \geq \pi + \sigma\frac{b_*^2}{L^2}.
  \end{equation}
  Estimates \eqref{eq43}, \eqref{eq57}, and \eqref{eq44} imply \eqref{eq45}.

  {\em Step 2.}
  Now we verify the main geometric estimate
  \begin{equation}\label{eq46}
    \frac{b_*^2}{L^2} + \frac{d_*-d}{L} \gtrsim \frac{a_*-a}{L},
  \end{equation}
  where $a:=z(L)=\int_0^L\sin\theta(s)ds$ and $a_*:=z_*(L)=z_\sharp(L)=\int_0^L|\sin\theta(s)|ds$.
  (Below we essentially use the hypothesis $a\geq0$.)
  Throughout this step we may assume that $L=1$ up to rescaling.
  In addition, we may assume that both $b_*\ll1$ and $d_*-d\ll1$ hold since otherwise \eqref{eq46} is trivial in view of $0\leq a_*-a \leq a_*\leq L=1$.
  For later use we introduce
  $$\bar{a}:=\max_{0\leq s_1<s_2\leq 1}\left|\int_{s_1}^{s_2}\sin\theta(s)ds\right|,$$
  which is nothing but the width in the axial direction $b_\Sigma(\omega_\mathrm{axis})$ of the original $\Sigma$.

  We first check the (optimal) estimate
  \begin{equation}\label{eq47}
    a_*-a \leq 2 (a_*-\bar{a}),
  \end{equation}
  by showing the equivalent one $2\bar{a}\leq a_*+a$.
  Notice the representation
  $$a_*+a = \int_0^1|\sin\theta(s)|ds + \int_0^1 \sin\theta(s)ds = 2 \int_0^1(\sin\theta(s))_+ds,$$
  where here and in the sequel we let $f_\pm\geq0$ denote the sign-decomposition $f=f_+-f_-$.
  Choose $s_1<s_2$ attaining the maximum in definition of $\bar{a}$, and let $J:=[s_1,s_2]$.
  Then we see that if $\alpha:=\int_J\sin\theta(s)ds\geq0$, then $\bar{a}=\int_J\sin\theta(s)ds$ and hence
  $$2\bar{a}=2\int_J\sin\theta(s)ds \leq 2\int_0^1(\sin\theta(s))_+ds = a_* + a,$$
  while if $\alpha \leq 0$, then $\bar{a}=-\int_J \sin\theta(s)ds$ and hence, noting that $a=\int_0^1\sin\theta(s)ds\geq0$, we also have
  $$2\bar{a}=-2\int_J \sin\theta(s)ds \leq 2\int_{[0,1]\setminus J}\sin\theta(s)ds \leq 2\int_0^1(\sin\theta(s))_+ds = a_*+a,$$
  completing the proof of \eqref{eq47}.

  We now prove that if $b_*\ll1$ and $d_*-d\ll1$, then
  \begin{equation}\label{eq48}
    a_*-\bar{a}\lesssim (d_*-d) + b_*^2.
  \end{equation}
  Since $\Sigma$ is contained in a cylinder of radius $b_*$ and height $\bar{a}$, we have
  \begin{equation}\label{eq54}
    d \leq \sqrt{\bar{a}^2+4b_*^2} \leq \bar{a} + 4\bar{a}^{-1}b_*^2.
  \end{equation}
  We now observe that the smallness assumptions imply
  \begin{equation}\label{eq56}
    \bar{a}\gtrsim1.
  \end{equation}
  Since $\gamma_*$ (generating $\Sigma_*$) is convex and of length $1$, we have $d_*\gtrsim1$.
  This and the assumption $d_*-d\ll1$ imply that $d\gtrsim1$.
  Using this and the assumption $b_*\ll1$ for \eqref{eq54}, we obtain $\bar{a}^2 \geq d^2-4b_*^2 \gtrsim1$ and hence \eqref{eq56} as desired.
  Inserting \eqref{eq56} into \eqref{eq54}, we get
  $$d -\bar{a} \lesssim b_*^2.$$
  Combining this with the obvious estimate $a_*\leq d_*$, we obtain \eqref{eq48}.
  Estimate \eqref{eq46} now follows from \eqref{eq47} and \eqref{eq48}.

  {\em Step 3.}
  We finally complete the proof.
  Estimates \eqref{eq45} and \eqref{eq46} imply that
  \begin{equation*}
    \frac{M(\Sigma)}{d} - \pi \gtrsim \frac{b_*^2}{L^2} + \frac{a_* - a}{L}.
  \end{equation*}
  Using the representations $2b_* = \int_0^L |\cos\theta(s)| ds$ and $a_*-a = 2\int_0^L (\sin\theta(s))_- ds$, and the change of variables $t=s/L$, we find that
  \begin{equation}\label{eq53}
    \frac{M(\Sigma)}{d} -\pi \gtrsim \Big(\int_0^1 |\cos\theta(Lt)| dt\Big)^2 + \int_0^1 (\sin\theta(Lt))_- dt.
  \end{equation}
  Simply letting $X^2+Y$ denote the right-hand side of \eqref{eq53}, we have $X^2+Y\gtrsim (X+Y)^2$ since $X,Y\in[0,1]$.
  In addition, the integrand of $X+Y$ has the lower bound of the form $|\cos\theta|+(\sin\theta)_-\gtrsim\sqrt{1-\sin\theta}$, where both sides linearly vanish if and only if $\theta\in\pi/2+2\pi\mathbb{Z}$.
  Therefore,
  \begin{equation*}
    \frac{M(\Sigma)}{d} -\pi \gtrsim \Big(\int_0^1 \sqrt{1-\sin\theta(Lt)} dt\Big)^2.
  \end{equation*}
  The relation $\sin\theta(Lt)=T(t)\cdot\omega_\mathrm{axis}$ for $\omega_\mathrm{axis}=(0,0,1)$ completes the proof.
\end{proof}

Here we discuss the optimality of Theorem \ref{thm:main1}.
As is mentioned in the introduction, the optimality can be observed by using the double-cone, which is nothing but the degenerate case $\Gamma^1_{0,\varepsilon}$ in Remark \ref{rem:examplecomputation}.

\begin{remark}[Optimality of Theorem \ref{thm:Topping} via the double-cone $\Gamma^1_{0,\varepsilon}$]\label{rem:doublecone}
  Let $\Sigma_\varepsilon$ be the thin double-cone $\Gamma^1_{0,\varepsilon}$ with $\varepsilon\ll1$.
  Obviously, $\di(\Sigma_\varepsilon)=1$.
  In addition, using the computation of $M$ in Remark \ref{rem:examplecomputation}, and in particular noting that the corresponding angle $\theta_\varepsilon$ in \eqref{eq52} is given by $\tan\theta_\varepsilon=2\varepsilon$ so that $\theta_\varepsilon/\varepsilon\to2$ as $\varepsilon\to0$, we have
  \begin{equation*}
    \frac{M(\Sigma_\varepsilon)}{\di(\Sigma_\varepsilon)} = \pi + 2\pi\theta_\varepsilon\varepsilon = \pi+4\pi\varepsilon^2+o(\varepsilon^2).
  \end{equation*}
  On the other hand, a simple computation yields
  $$U(\Sigma_\varepsilon) = \Big(\int_0^1\sqrt{1-\cos\theta_\varepsilon}\Big)^2 = 1-\cos\theta_\varepsilon =  2\varepsilon^2+o(\varepsilon^2),$$
  and hence the ratio $\frac{1}{U(\Sigma_\varepsilon)}(\frac{M(\Sigma_\varepsilon)}{\di(\Sigma_\varepsilon)}-\pi)$ converges as desired.
  Note that by a suitable approximation of higher order than $\varepsilon$, we can even construct a sequence of smooth surfaces (nearly double-cone) for which the ratio also converges.
\end{remark}

We additionally remark a difference between the axial and coaxial directions.

\begin{remark}[Optimal remainder in the axial direction]\label{rem:optimalityaxial}
  In the third step of the proof of Theorem \ref{thm:Topping}, the remainder is eventually simplified into $U$, but this is just for notational convenience to state Theorem \ref{thm:Topping}.
  The simplified remainder $U$ is still sharp with respect to coaxial deviation as in Remark \ref{rem:doublecone}, but in fact not sharp axially.
  To see this fact, we let $\Sigma_{\varepsilon}$ be the surface generated by revolving the broken line connecting $(x,z)=(0,0),(\varepsilon^2,-\varepsilon),(\varepsilon^2,1+\varepsilon),(0,1)$ around the $z$-axis.
  Then we have
  $$\frac{M(\Sigma_\varepsilon)}{\di(\Sigma_\varepsilon)}=(1+4\varepsilon)\pi + o(\varepsilon), \quad U(\Sigma_\varepsilon) = 4\varepsilon^2+o(\varepsilon^2),$$
  so that the ratio $\frac{1}{U}(\frac{M}{\di}-\pi)$ diverges.
  However, if we replace $U$ by a more sharp remainder $V$ such as the right-hand side in \eqref{eq53}, i.e.,
  $$V(\Sigma) := \Big(\int_0^1 |N(t)\cdot\omega_\mathrm{axis}| dt\Big)^2 + \int_0^1 (T(t)\cdot\omega_\mathrm{axis})_- dt,$$
  where $N(t)$ denotes the unit normal of $\Sigma$ at $\gamma_\Sigma(t)$, then the ratio $\frac{1}{V}(\frac{M}{\di}-\pi)$ converges even for the above $\Sigma_\varepsilon$.
\end{remark}

We now complete the proof of Corollary \ref{cor:Topping}.

\begin{proof}[Proof of Corollary \ref{cor:Topping}]
  We only need to argue for a minimizing sequence $\{\Sigma_n\}$.
  We may assume up to rescaling that $\di(\Sigma_n)=1$, and up to a rigid motion that $\Sigma_n$ satisfies Hypothesis \ref{hyp:axisymmetric}.
  Let $\theta_n:[0,1]\to\mathbb{R}$ denote the angle function in Hypothesis \ref{hyp:axisymmetric} corresponding to $\Sigma_n$ after the change of variables $t=s/L_n$, where $L_n$ is the length of a generating curve of $\Sigma_n$.
  Since the distance of the two points in the $z$-axis (i.e., the endpoints of a generating curve) is bounded above by the diameter, we have
  \begin{equation}\label{eq49}
    L_n\int_0^1\sin\theta_n(t)dt \leq 1.
  \end{equation}
  The assumption of convergence $M(\Sigma_n)/\di(\Sigma_n)\to\pi$ and Theorem \ref{thm:Topping} imply that $\int_0^1\sqrt{1-\sin\theta_n}\to0$; hence, $\sin\theta_n(t)\to1$ and also $\cos\theta_n(t)\to0$ for a.e.\ $t\in[0,1]$.
  Using the bounded convergence theorem, we deduce that $\sin\theta_n\to1$ and $\cos\theta_n\to0$ in $L^p((0,1))$ for any $1\leq p < \infty$.
  In addition, also noting that $L_n\leq1$ ($=\di(\Sigma_n)$), we deduce from \eqref{eq49} that $L_n\to1$.
  Hence, for the derivative $\dot{\gamma}_{\Sigma_n}=L_n(\cos\theta_n,0,\sin\theta_n)$ of the generating curve $\gamma_{\Sigma_n}$ chosen to lie in the $xz$-plane, we see that $\dot{\gamma}_{\Sigma_n}$ converges in $L^p((0,1);\mathbb{R}^3)$ to the derivative $\dot{\bar{\gamma}}=(0,0,1)$ of the segment $\bar{\gamma}(t)=(0,0,t)$.
  The lower order convergence of $\gamma_{\Sigma_n}$ to $\bar{\gamma}$ easily follows from this first order one.
\end{proof}

We finally recall that Topping's conjecture is also related to finding the optimal constant in Simon's inequality of the form
$$A^\frac{1}{2}\Big(\int_\Sigma H^2\Big)^\frac{1}{2}\geq\frac{\pi}{2}\di(\Sigma).$$
The constant $\pi/2$ is explicitly obtained by Topping \cite{Topping1998} following Simon's original strategy \cite{Simon1993}.
It is also conjectured that $\pi/2$ can be replaced with $\pi$ by the same reason as Topping's conjecture.
Since Simon's inequality is implied by Topping's inequality via the Cauchy-Schwarz inequality, our result also gives the optimal constant for Simon's one under the same assumption as in Theorem \ref{thm:Topping}.

\appendix

\section{Mean curvature estimates in rearrangements}\label{app:meancurvature}

We briefly recall the arguments in \cite{Dalphin2016} about how the rearrangements defined in Section \ref{sec:Topping} control mean curvature.

We first address the first rearrangement, $\Sigma\to\Sigma_\sharp$, and prove that $\int_{\Sigma_\sharp}|H|=\int_\Sigma|H|$.
A direct computation yields the representation
\begin{equation}\label{eq28}
  \int_\Sigma|H|=\pi\int_0^L|\sin\theta(s)+\theta_s(s)x(s)|ds.
\end{equation}
As is already observed in the above proof, we have
\begin{equation}\label{eq27}
  x_\sharp(s)=
  x(s) \quad \mbox{for every} \ s\in[0,L],
\end{equation}
and in addition that $|\sin\theta|=\sin\theta^\sharp$; more precisely,
\begin{align}\label{eq25}
  \sin\theta^\sharp(s) = \pm\sin\theta(s) \quad \mbox{if}\ \theta(s)\in S_\pm:= 2\pi\mathbb{Z}\pm[0,\pi].
\end{align}
Furthermore, since $\theta^\sharp=f\circ\theta$, where both $f:=\mathrm{dist}(\cdot,2\pi\mathbb{Z})$ and $\theta$ are Lipschitz, we have the chain rule $\theta^\sharp_s=(f'\circ\theta)\theta_s$ with the understanding that $[(f'\circ\theta)\theta_s](s)=0$ whenever $\theta_s(s)=0$ (cf.\ \cite[Corollary 3.66]{Leoni2017}).
Since $f'\equiv\pm1$ on $S_\pm$, we get
\begin{equation}\label{eq26}
  \theta^\sharp_s(s) = \pm\theta_s(s) \quad \mbox{if}\ \theta(s)\in S_\pm.
\end{equation}
Inserting \eqref{eq27}, \eqref{eq25}, and \eqref{eq26} into \eqref{eq28}, we deduce that
$\int_{\Sigma_\sharp}|H|=\int_{\Sigma}|H|$.

We now turn to the second rearrangement, $\Sigma\to\Sigma_*$, and prove that
$\int_{\Sigma_*}|H|\leq\int_\Sigma|H|$.
Since $\int_{\Sigma_*}|H|=\int_{\Sigma_*}H$ by convexity of $\Sigma_*$, in view of the triangle inequality it suffices show that $\int_{\Sigma_*}H=\int_\Sigma H$.
By a direction computation and an integration by parts, using that $x(0)=x(L)=0$, we have the representation
\begin{equation*}
  \int_\Sigma H = \pi\int_0^L \big( \sin\theta(s)+\theta_s(s)x(s) \big)ds = \pi\int_0^L g(\theta(s))ds,
\end{equation*}
where $g(\theta):= \sin\theta -\theta\cos\theta$.
The same representation $\int_{\Sigma_*}H=\pi\int_0^Lg(\theta^*(s))ds$ also holds for the rearranged surface $\Sigma_*$ since $x_*(0)=x_*(L)=0$.
We then deduce the desired identity $\int_{\Sigma_*}H=\int_\Sigma H$
from these representations and the integration-preserving property of rearrangement.

\bibliography{bibliography}
\bibliographystyle{amsplain}

\end{document}